\documentclass[12pt]{amsart}

\usepackage{amssymb}
\usepackage{enumerate}
\usepackage{subfigure}

\numberwithin{equation}{section}

\newtheorem{thm}{Theorem}[section]

\newtheorem{lem}[thm]{Lemma}

\theoremstyle{definition}
\newtheorem{defn}{Definition}[section]
\newtheorem{ex}[defn]{Example}
\newtheorem{rmk}[defn]{Remark}

\newcommand{\chr}{\mathrm{char}}

\newcommand{\sign}{\text{sgn}}
\newcommand{\triv}{\text{triv}}

\newcounter{countcases}

\newcommand{\PBS}[1]{\let\temp=\\#1\let\\=\temp}  

\numberwithin{figure}{section}

\begin{document}
\author{Brendon Rhoades}
\email{brhoades@math.mit.edu}
\address{Brendon Rhoades, Department of Mathematics, Massachusetts Institute
of Technology, Cambridge, MA, 02139}
\title[Hall-Littlewood polynomials and fixed point enumeration]
{Hall-Littlewood polynomials and fixed point enumeration}


\bibliographystyle{../dart}

\date{\today}

\begin{abstract}
We resolve affirmatively some conjectures of 
Reiner, Stanton, and White \cite{ReinerComm} regarding 
enumeration of transportation matrices which are 
invariant under certain cyclic row and column rotations.
Our results are phrased in terms of the bicyclic sieving
phenomenon introduced by Barcelo, Reiner, and Stanton \cite{BRSBiD}.
The proofs of our results use various tools from symmetric 
function theory such as the Stanton-White rim hook correspondence
\cite{SW} and results concerning the specialization of 
Hall-Littlewood polynomials due to Lascoux, Leclerc, and 
Thibon \cite{LLTUnity} \cite{LLTRibbon}.
\end{abstract}
\maketitle

\section{Introduction and Main Results}\label{s:intro}

Let $X$ be a finite set and $C \times C'$ be a direct product 
of two finite cyclic groups acting on $X$.  Fix generators $c$ and $c'$ for $C$ and $C'$ and 
let $\zeta, \zeta' \in \mathbb{C}$ be two roots of unity
having the same multiplicative orders as $c, c'$.
Let
$X(q,t) \in \mathbb{C}[q,t]$ be a polynomial in two 
variables.  
Following
Barcelo, Reiner, and Stanton \cite{BRSBiD}, we say that the triple
$(X, C \times C', X(q,t))$ \emph{exhibits the bicyclic
sieving phenomenon} (biCSP) if for any integers 
$d, e \geq 0$ the cardinality of the fixed point set
$X^{(c^d,c'^e)}$ is equal to the polynomial evaluation
$X(\zeta^d,\zeta^e)$.  
The biCSP encapsulates several combinatorial phenomena:
specializing to the case where one of the cyclic groups is
trivial yields the \emph{cyclic sieving phenomenon} of 
Reiner, Stanton, and White \cite{RSWCSP} and specializing further
to the case where the nontrivial cyclic group has order 
two yields the \emph{$q = -1$ phenomenon} of Stembridge 
\cite{StemTab}.  
Moreover, the fact that the identity element in any group
action fixes everything implies that whenever 
$(X, C \times C', X(q,t))$ exhibits the biCSP,
we must have that the $q = t = 1$ specialization
$X(1,1)$ is equal to the cardinality $|X|$ of the set $X$.
In this paper we
prove a pair of biCSPs conjectured by 
Reiner, Stanton, and White where the sets $X$ are certain 
sets of matrices acted on by row and column rotation and the 
polynomials $X(q,t)$ are bivariate deformations of 
identities arising from the RSK insertion algorithm.
Our proof, outlined in Section 2, relies on symmetric function theory 
and plethystic substitution.
In Section 3 we outline an alternative 
argument due to Victor Reiner which proves 
these biCSPs `up to modulus' using DeConcini-Procesi modules.

Given a partition $\lambda \vdash n$, recall that a
\emph{semistandard Young tableau (SSYT) of shape $\lambda$} is
a filling of the Ferrers diagram of $\lambda$ with 
positive numbers which increase strictly down columns and 
weakly across rows.  For a SSYT $T$ of shape $\lambda$, the 
\emph{content} of $T$ is the (weak) composition
$\mu \models n$ given by letting $\mu_i$ equal the number
of $i's$ in $T$.  
A SSYT $T$ is called \emph{standard} (SYT) if it has content 
$1^n$.
For a partition 
$\lambda$ and a composition $\mu$ of $n$, 
the \emph{Kostka number} 
$K_{\lambda,\mu}$ is equal to the number of SSYT of shape
$\lambda$ and content $\mu$.  

The \emph{Kostka-Foulkes polynomials} $K_{\lambda,\mu}(q)$, indexed by
a partition $\lambda \vdash n$ and a composition 
$\mu \models n$, arose originally as the entries of the transition
matrix between the Schur function and Hall-Littlewood symmetric
function bases of the ring of symmetric functions (with coefficients
in $\mathbb{C}(q)$ where $q$ is an indeterminate).  A combinatorial
proof of the positivity of their coefficients was given by Lascoux
and Sch\"utzenberger \cite{LSFoulkes} by identifying $K_{\lambda,\mu}(q)$ as 
the generating function for the statistic of \emph{charge} on the set
of semistandard tableaux of shape $\lambda$ and content $\mu$. 
We outline the definition of charge as the rank function of a
cyclage poset.
 
Let $\mathcal{A}^{*}$ denote the free monoid of words $w_1 \dots w_k$ of any
length with letters drawn from $[n]$.  Let $\equiv$ be the
equivalence relation on $\mathcal{A}^{*}$ induced by 
$
\begin{array}{cccc}
RkijR' \equiv RikjR', & RjikR \equiv RjkiR', & RjiiR' \equiv RijiR', &  RjijR' \equiv RjjiR',
\end{array}
$
\noindent
where $1 \leq i < j < k \leq n$ and $R$ and $R'$ are any words in the
monoid $\mathcal{A}^{*}$.  
The \emph{Robinson-Schensted-Knuth correspondence} yields
an algorithmic bijection between words $w$ in $\mathcal{A}^{*}$ and pairs
$(P(w),Q(w))$ of tableaux, where $P$ is a SSYT with entries
$\leq n$ and $Q$ is a SYT with
the shape of $P(w)$ equal to the shape of $Q(w)$.  For details on the RSK
correspondence, see for example \cite{Sag} or \cite{StanEC2}.  
The RSK correspondence sets up an equivalence relation $\equiv'$ on words in 
$\mathcal{A}^{*}$ by setting 
$w \equiv' w'$ if and only if $Q(w) = Q(w')$.  
It is a result of Knuth \cite{KnuthPerm} that the equivalence relations $\equiv$ and
$\equiv'$ on $\mathcal{A}^{*}$ agree.  
That is, for any $w, w' \in \mathcal{A}^{*}$ we have
$w \equiv w'$ if and only if $Q(w) = Q(w')$.  Therefore, the quotient monoid
$\mathcal{A}^{*}/\equiv$ is in a natural bijective correspondence 
with the set of
SSYT with entries $\leq n$.  This quotient is called the
\emph{plactic monoid}.  

\emph{Cyclage} is a monoid analogue of the
group operation of conjugation introduced by
Lascoux and Sch\"utzenberger \cite{LS7}.  
Given $w, w' \in \mathcal{A}^{*}/\equiv$, say that $w \prec w'$ if there exists $i \geq 2$ and $u \in \mathcal{A}^{*}/\equiv$ 
so that $w = iu$ and $w' = ui$.  
For a fixed composition $\mu \models n$, the transitive closure of the relation
$\prec$ induces a partial order on the subset 
of $\mathcal{A}^{*}/\equiv$ consisting
of words of content $\mu$, and therefore also on the set of SSYT of
content $\mu$.  For fixed $\mu$, the rank generating function for this
poset is called \emph{cocharge} and is therefore a statistic on 
SSYT of content $\mu$.  The rank function of the order theoretic dual of this poset is called \emph{charge}.  Lascoux and Sc\"utzenberger
\cite{LSFoulkes} proved that for any partition $\lambda \vdash n$ and any composition
$\mu \models n$, we have that            
\begin{equation*}
K_{\lambda,\mu}(q) = \sum_{T} q^{charge(T)},
\end{equation*} 
where the sum ranges over all SSYT $T$ of shape $\lambda$
and content $\mu$.

For $n \geq 0$, define $\epsilon_n(q,t) \in \mathbb{N}[q,t]$ to be 
$(qt)^{n/2}$ if $n$ is even and 1 if $n$ is odd.
The type $A$ specialization of Theorem 1.4 of Barcelo,
Reiner, and Stanton \cite{BRSBiD} yields the following:

\begin{thm} (\cite{BRSBiD})
Let $X$ be the set of $n \times n$ permutation matrices and
$\mathbb{Z}_n \times \mathbb{Z}_n$ act on $X$ by row and 
column rotation.  The triple 
$(X, \mathbb{Z}_n \times \mathbb{Z}_n, X(q,t))$ exhibits 
the biCSP, where
\begin{equation*}
X(q,t) = \epsilon_n(q,t)
\sum_{\lambda \vdash n} K_{\lambda,1^n}(q) K_{\lambda,1^n}(t).
\end{equation*}
\end{thm}

\begin{ex}
Let $n = 4$.  We have that 
\begin{align*}
X(q,t)  =  (qt)^2 \Big[&(qt)^6 + (qt)^3(1+q+q^2)(1+t+t^2) + (qt)^2(1+q^2)(1+t^2)  \\
                    &+ (qt)(1+q+q^2)(1+t+t^2) + 1 \Big] .
\end{align*}
Consider the action of the diagonal subgroup of $\mathbb{Z}_4 \times \mathbb{Z}_4$ on $X = S_4$.  Let $r$ be the generator of this subgroup, so that $r$ acts on $X$ by a simultaneous single row and column shift.  We have that $X(i,i) = 4$, reflecting the fact that the fixed point set 
$X^r = \{ 1234, 2341, 3412, 4123 \}$ has four elements.  Also, $X(-1,-1) = 8$, whereas the fixed point set $X^{r^2} = \{ 1234, 2341, 3412, 4123, 1432, 2143, 3214, 4321 \}$.  Finally, we have that 
$X(i,-1) = 0$, reflecting the fact that no $4 \times 4$ permutation matrix is fixed by a simultaneous $1$-fold row shift and $2$-fold column shift.
\end{ex}

The $q = t = 1$ specialization of $X(q,t)$ in the above
result is implied by the RSK insertion algorithm on
permutations.  The following generalization of Theorem 1.1 to the case of words was known to
Reiner and White but is unpublished.  
For any composition $\mu \models n$, let 
$\ell(\mu)$ denote the number of parts of $\mu$ and
$|\mu| = n$ denote the sum of the parts of $\mu$.
A composition $\mu \models n$ is
said to have cyclic symmetry of order $a$ if one has
$\mu_i = \mu_{i+a}$ always, where subscripts are 
interpreted modulo $\ell(\mu)$.

\begin{thm} (\cite{ReinerComm}, \cite{WComm})
Let $\mu \models n$ be a composition with cyclic
symmetry of order $a | \ell(\mu)$.  Let $X$ be the set of
length $n$ words of content $\mu$, thought of as
0,1-matrices in the standard way.  The product of 
cyclic groups $\mathbb{Z}_{\ell(\mu)/a} \times \mathbb{Z}_n$
acts on $X$ by $a$-fold row rotation and 1-fold column
rotation.  

The triple 
$(X, \mathbb{Z}_{\ell(\mu)/a} \times \mathbb{Z}_n,
X(q,t))$ exhibits the biCSP, where
\begin{equation*}
X(q,t) = \epsilon_n(q,t)
\sum_{\lambda \vdash n} K_{\lambda,\mu}(q) K_{\lambda,1^n}(t).
\end{equation*} 
\end{thm}

\begin{ex}
Let us give an example to show why the factor $\epsilon_n(q,t)$ is necessary
in the definition of $X(q,t)$.  Take $n = 2, \mu = (2),$ and $a = 1$.  The 
set $X$ is the singleton 
$\{11\}$
consisting of the word $11$.  One verifies that
$
\begin{array}{cccc}
K_{(1,1),(2)}(q) = 0, &K_{(2),(2)}(q) = 1, &K_{(1,1),(2)}(t) = 1, 
&K_{(2),(2)}(t) = t,
\end{array}
$
\noindent
so that
\begin{equation*}
X(q,t) = (qt) [ 0(1) + 1(t) ] = qt^2.
\end{equation*}
We have the evaluation $X(1,-1) = 1$, which would have been negative if
$X(q,t)$ did not contain the factor of $\epsilon_2(q,t) = qt$.  
\end{ex}

The $q = t = 1$ specialization of the identity in the 
above theorem arises from the application of RSK to
the set of words with content $\mu$.  
The following
$\mathbb{N}$-matrix generalization 
of Theorem 1.2 was conjectured (unpublished)
by Reiner and White in 2006.

\begin{thm}
Let $\mu, \nu \models n$ be two compositions having 
cyclic symmetries of orders $a | \ell(\mu)$ and 
$b | \ell(\nu)$, respectively.  Let $X$ be the set of 
$\ell(\mu) \times \ell(\nu)$ $\mathbb{N}$-matrices with
row content $\mu$ and column content $\nu$.
The product of cyclic groups 
$\mathbb{Z}_{\ell(\mu)/a} \times \mathbb{Z}_{\ell(\nu)/b}$
acts on $X$ by $a$-fold row rotation and $b$-fold
column rotation.  

The triple $(X,\mathbb{Z}_{\ell(\mu)/a} \times \mathbb{Z}_{\ell(\nu)/b},
X(q,t))$ exhibits the biCSP, where
\begin{equation*}
X(q,t) = \epsilon_n(q,t) 
\sum_{\lambda \vdash n} K_{\lambda,\mu}(q) K_{\lambda,\nu}(t).
\end{equation*} 
\end{thm}

As before, the $q = t = 1$ specialization of the above identity
follows from applying RSK to the set $X$.  The `dual Cauchy'
version of the previous result which follows was suggested
by Dennis Stanton after the author's thesis defense.

For any $n > 0$, let $\delta_n(q,t) \in \mathbb{C}[q,t]$ be a polynomial
whose evaluations $\delta_n(\zeta,\zeta')$ at $n^{th}$ roots of unity
$\zeta, \zeta'$ with multiplicative orders $|\zeta| = k$ and
$|\zeta'| = \ell$ satisfy
\begin{equation*}
\delta_n(\zeta,\zeta') = \begin{cases}
1 & \text{if $\frac{n}{k}$ and $\frac{n}{\ell}$ are even,} \\
1 & \text{if $k$ and $\ell$ are odd,} \\
-1 & \text{if $k, \ell$ are even and $\frac{n}{k}, \frac{n}{\ell}$ are odd,} \\
-1 & \text{if exactly one of $\frac{n}{k},\frac{n}{\ell}$ is even and
both $k, \ell$ are even,}\\
1 & \text{if exactly one of $\frac{n}{k}, \frac{n}{\ell}$ is even and
exactly one of $k, \ell$ is even.}
\end{cases}
\end{equation*} 
An explicit formula for a choice of $\delta_n(q,t)$ can be found using 
Fourier analysis on the direct product $\mathbb{Z}_n \times \mathbb{Z}_n$ of
cyclic groups, but the formula so obtained is somewhat messy.  
It should be noted that if $n$ is odd, one can take $\delta_n(q,t) \equiv 1$.  

\begin{thm}
Let $\mu, \nu \models n$ be two compositions having 
cyclic symmetries of orders $a | \ell(\mu)$ and 
$b | \ell(\nu)$, respectively.  Let $X$ be the set of 
$\ell(\mu) \times \ell(\nu)$ $0,1$-matrices with
row content $\mu$ and column content $\nu$.
The product of cyclic groups $\mathbb{Z}_{\ell(\mu)/a} \times \mathbb{Z}_{\ell(\nu)/b}$
acts on $X$ by $a$-fold row rotation and $b$-fold column rotation.

The triple $(X, \mathbb{Z}_{\ell(\mu)/a} \times \mathbb{Z}_{\ell(\nu)/b}, X(q,t))$ exhibits the biCSP, where
$X(q,t) \in \mathbb{C}[q,t]$ is
\begin{equation*}
X(q,t) = \delta_n(q,t)
\sum_{\lambda \vdash n} K_{\lambda',\mu}(q) K_{\lambda,\nu}(t).
\end{equation*}
\end{thm}    

\begin{ex}
Let us give an example to show why the factor of $\delta_n(q,t)$ is necessary in the statement of
Theorem 1.4.  Take $n = 2$, $\mu = \nu = (1,1)$, and $a = b = 1$.  The set $X$ can be identified
with the two permutation matrices in $S_2$.  The polynomial $X(q,t)$ is given by
$X(q,t) = \delta_2(q,t) (q+t)$ and the evaluation
$X(-1,-1) = \delta_2(-1,-1) (-1-1) = (-1)(-2) = 2$ would have been negative without
the factor $\delta_2(-1,-1)$.
\end{ex}

The $q = t = 1$ specialization of Theorem 1.4 
follows from applying the \emph{dual} RSK algorithm 
to the set $X$ (see \cite{StanEC2}).  By the definition of 
$\delta_n(q,t)$, we have that
$\delta_n(q,t) \in \{1, -1 \}$ whenever $q$ and $t$ are specialized to 
$n^{th}$ roots of unity.  Therefore, omitting the factor
$\delta_n(q,t)$ in Theorem 1.4 gives a biCSP `up to sign'.

\begin{rmk}
Given a finite set $X$ acted on by a finite product $C \times C'$ of cyclic
groups, it is always possible to find some polynomial $X(q,t)$ such that
the triple $(X, C \times C', X(q,t))$ exhibits the biCSP.  The interest 
in a biCSP lies in giving a polynomial $X(q,t)$ with a particularly nice form,
either as an explicit product/sum formula or as a generating function for
some pair of natural combinatorial statistics on the set $X$.  
We observe that, apart from the factors of $\epsilon_n(q,t)$ and
$\delta_n(q,t)$, our polynomials $X(q,t)$ are nice in this latter sense.

Indeed, one can represent any $\mathbb{N}$-matrix $A$ whose entries sum to $n$
as a $2 \times n$ matrix 
$\begin{pmatrix}
a_{11} & a_{12} & \dots & a_{1n} \\
a_{21} & a_{22} & \dots & a_{2n}
\end{pmatrix}$,
where the biletters 
$\begin{pmatrix}
a_{1i} \\
a_{2i} \end{pmatrix}$ are in lexicographical order and the biletter
$\begin{pmatrix}
i \\
j  \end{pmatrix}$ occurs with multiplicity equal to the $(i,j)$-entry of 
$A$.  The word $w_A := a_{21} a_{22} \dots a_{2n}$ given by the bottom
row of this matrix is mapped to a pair $(P(w_A), Q(w_A))$ under RSK insertion,
where $P(w_A)$ is a semistandard tableau of content equal to the column
content vector of the matrix $A$ and $Q(w_A)$ is a standard tableau having
the same shape as $P(w_A)$.  Using the fact that matrix transposition
corresponds under RSK to swapping tableaux (see \cite{StanEC2}), 
one sees that for any compositions
$\mu, \nu \models n$,
\begin{equation*}
\sum_{\lambda \vdash n} K_{\lambda,\mu}(q) K_{\lambda,\nu}(t) =
\sum_A q^{charge(w_{A^T})} t^{charge(w_A)},
\end{equation*}
where the sum ranges over the set of all $\mathbb{N}$-matrices $A$ with row 
content $\mu$ and column content $\nu$
and $A^T$ is the transpose of $A$.  Similarly, one has that
\begin{equation*}
\sum_{\lambda \vdash n} K_{\lambda,\mu}(q) K_{\lambda',\nu}(t) =
\sum_A q^{charge(w_{A^T})} t^{charge(w_A)},
\end{equation*}
where the sum ranges over all 0,1-matrices $A$ with row content $\mu$ and 
column content $\nu$.  Thus, apart from the factors 
$\epsilon_n(q,t)$ and $\delta_n(q,t)$, the polynomials
$X(q,t)$ appearing in the biCSPs of Theorems 1.3 and 1.4 are the 
generating functions for the pair of statistics
$A \mapsto (charge(A^T), charge(A))$ on the set $X$. 
\end{rmk}
\section{Proofs of Theorems 1.3 and 1.4}

The proofs of all of the above biCSPs will be `semi-combinatorial',
relying on enumerative results arising from RSK and
the Stanton-White rim hook correspondence \cite{SW} as well as
algebraic results from symmetric function theory due to 
Lascoux, Leclerc, and Thibon \cite{LLTUnity} \cite{LLTRibbon}.
Interestingly, although the formulas for $X(q,t)$ involve many 
Kostka-Foulkes polynomials, we shall not explicitly need any facts
about the charge statistic on tableaux.
Let $\Lambda$ denote the ring of symmetric functions
in $x = (x_1, x_2, \dots)$ having coefficients in 
$\mathbb{C}(q)$, where $q$ is a formal indeterminate.  
The \emph{Hall inner product} 
$\langle \cdot , \cdot \rangle$
on $\Lambda$ defined by declaring the basis 
$\{s_{\lambda}\}$ of Schur functions to be orthonormal.
For any composition $\mu \models n$, the \emph{Hall-Littlewood
symmetric function} $Q_{\mu}(x_1,x_2, \dots ; q)$ is defined
by
\begin{equation*}
Q_{\mu}(x;q) = \sum_{\lambda \vdash n} K_{\lambda,\mu}(q)s_{\lambda}(x).
\end{equation*}
Specializing to $q = 1$, we have that $Q_{\mu}(x;1) = \sum_{\lambda \vdash n} K_{\lambda, \mu} s_{\lambda}(x) = h_{\mu}(x)$, where $h_{\mu}$ is the complete homogeneous symmetric function indexed by $\mu$.  Thus, the Hall-Littlewood symmetric functions may be regarded
as $q$-deformations of the homogeneous symmetric functions.

For any 
$k \geq 0$, define a linear operator $\psi^{k}$ on $\Lambda$
by
\begin{equation*}
\psi^k(F(x_1, x_2, \dots)) = p_k \circ F = 
F(x_1^k, x_2^k, \dots).
\end{equation*}
Here $p_k \circ F$ is plethystic substitution.
Following Lascoux, Leclerc, and Thibon \cite{LLTUnity}, let 
$\phi_k$ be the adjoint of $\psi^k$ with respect to the 
Hall inner product.  That is, $\phi_k$ is defined by the 
condition $\langle F, \phi_k(G) \rangle = 
\langle \psi^k(F), G \rangle$ for any symmetric functions
$F, G$.  
For any composition $\mu \models n$ and any positive integer 
$k$ so that $\mu_i | k$ for all $i$, define the composition
$\frac{1}{k} \mu \models n/k$ by 
$(\frac{1}{k} \mu)_i = \frac{\mu_i}{k}$.
In addition, for any composition $\mu \models n$ with all 
part multiplicities divisible by $k$,  let $\mu^{1/k}$be any composition of $n/k$ obtained by 
dividing all part multiplicities 
in $\mu$ by $k$.  
In particular, if all of the part multiplicities in $\mu$ are divisible by 
$k$, the power sum symmetric function 
$p_{\mu^{1/k}}$, the elementary symmetric function
$e_{\mu^{1/k}}$, and the complete homogeneous symmetric 
function $h_{\mu^{1/k}}$ are all well-defined.
Finally, let $\omega: \Lambda \rightarrow \Lambda$
be the involution on the ring of symmetric functions which
interchanges elementary and homogeneous symmetric functions:
$\omega(e_n) = h_n$. 

\begin{lem}
The operators $\psi^{k}$ and $\phi_k$ are both ring
homomorphisms.  Moreover, we have the following 
equalities of operators 
on $\Lambda$ for any $k, \ell \geq 0$.\\
1. $\psi^{k} \psi^{\ell} = \psi^{k \ell}$ \\
2. $\phi_k \phi_{\ell} = \phi_{k \ell}$ \\
3. $\phi_k \psi^{k} \phi_{\ell} = 
    \phi_{\ell} \phi_k \psi^{k}$ \\
4. $\phi_k \psi^k \psi^{\ell} =
    \psi^{\ell} \phi_k \psi^k$. \\
If in addition $k$ and $\ell$ are relatively prime, 
we also have \\
5. $\phi_k \psi^{\ell} = \psi^{\ell} \psi_k$.
\end{lem}  

\begin{proof}
Clearly $\psi^k$ is a ring map.
Using the fact that 
$\phi_k$ is the adjoint to
$\psi^k$, it's easy to check that we have the following formula
for $\phi_k$ evaluated on power sum symmetric functions
$p_{\mu}$ for $\mu \models n$:
\begin{equation*}
\phi_k(p_{\mu}) = k^{\ell(\mu)} p_{\mu/k}.
\end{equation*}

Here we interpret the right hand side to be 0 if $k$ does not 
divide every part of $\mu$.
From this formula it follows that $\phi_k$ is 
a ring homomorphism.  Now relations 1 through 5 can be 
routinely checked on the generating set
$\{ p_n \}$ of $\Lambda$ given by power sums.
\end{proof}

Remarkably, the operators $\psi^k$ can be used to evaluate
certain specialized Hall-Littlewood polynomials.  The specializations
involve application of the raising operators $\psi^k$ to homogeneous
symmetric functions.  Recall that a composition $\mu \models n$ is 
\emph{strict} if all of its parts are strictly positive.

\begin{thm} (Lascoux-Leclerc-Thibon \cite[Theorems 3.1, 3.2]{LLTUnity})
Let $\mu \models n$ be a strict composition and for 
$k | n$ 
let $\zeta$ be a primitive $k^{th}$ root of unity.
Assume that all the part multiplicities in $\mu$ are 
divisible by $k$.  Then, we have 
\begin{equation*}
Q_{\mu}(x;\zeta) = (-1)^{(k-1)\frac{n}{k}}
\psi^k (h_{\mu^{1/k}}).
\end{equation*}
\end{thm}     

The sign appearing in the above theorem is the reason
why we needed the factor of $\epsilon_n(q,t)$ in 
Theorem 1.3.

\begin{proof} (of Theorem 1.3) 
Without loss of generality we may assume that the compositions
$\mu$ and $\nu$ are strict.
Let $\zeta$ and $\zeta'$ be roots of unity of 
multiplicative orders $k$ and $\ell$, where each part
of $\mu$ has multiplicity divisible by $k$ and 
each part of $\nu$ has multiplicity divisible by $\ell$.
Temporarily ignoring the factor 
$\epsilon_n(q,t)$, we are interested in expressions like
\begin{equation*}
\sum_{\lambda \vdash n} 
K_{\lambda,\mu}(\zeta) K_{\lambda,\nu}(\zeta').
\end{equation*}  
This sum
is equal to the Hall inner product 
\begin{equation*}
\langle Q_{\mu}(x ; \zeta), Q_{\nu} (x ; \zeta') \rangle
\end{equation*}
of specialized Hall-Littlewood functions.

By Theorem 2.2, the above inner product up to sign is equal
to
\begin{equation*}
\langle \psi^k (h_{\mu^{1/k}}),
\psi^{\ell} (h_{\nu^{1/\ell}}) \rangle.
\end{equation*}
Let $m$ be the greatest common divisor of $k$ and $\ell$.
Applying the operator calculus in Lemma 2.1, we see that
the previous inner product is equal to
\begin{equation*}
\langle \psi^m \phi_{\ell/m} (h_{\mu^{1/k}}) ,
\psi^m \phi_{k/m} (h_{\nu^{1/\ell}}) \rangle.
\end{equation*}

For $N \geq 0$, recall that the $a$-core of the partition $(N)$ with a single
part is empty if and only if $a | N$, in which case the $a$-quotient of $(N)$ is the sequence
$((\frac{N}{a}), \emptyset , \dots, \emptyset )$ and the $a$-sign of $(N)$ is 1.
By a result of Littlewood \cite{LittlewoodMod} (See Formula 13 of
\cite{LLTRibbon}), 
 the evaluation $\phi_a(h_N)$
is equal to $h_{N/a}$ if $a | N$ and 0 otherwise.    
From this and the fact that the $\phi$ operators are ring
homomorphisms we get that the last inner product is equal
to
\begin{equation*}
\langle \psi^m (h_{\frac{m}{\ell} \mu^{1/k}}),
\psi^m (h_{\frac{m}{k} \nu^{1/\ell}}) \rangle,
\end{equation*} 
where we interpret $h_{\frac{1}{a} \alpha}$ to be equal to
zero if every part size in $\alpha$ is not divisible by $a$.

Formula 17 in \cite{LLTRibbon} implies that
\begin{equation*}
\psi^m (h_{\alpha} ) = \sum_T \epsilon_m(T) s_{sh(T)},
\end{equation*}
where the sum ranges over all semistandard $m$-ribbon
tableaux $T$ having content $\alpha$, 
$\epsilon_m(T)$ is the $m$-sign of the ribbon tableau
$T$, and sh($T$) is the shape of $T$.  By the orthonormality
of the Schur function basis, this implies that the 
inner product of interest
\begin{equation*}
\langle \psi^m (h_{\frac{m}{\ell} \mu^{1/k}}),
\psi^m (h_{\frac{m}{k} \nu^{1/\ell}}) \rangle,
\end{equation*}
is equal to the number of ordered pairs $(P,Q)$ of 
semistandard $m$-ribbon tableaux 
of the same shape
where $P$ has content
$\frac{m}{\ell} \mu^{1/k}$ and $Q$ has content
$\frac{m}{k} \nu^{1/\ell}$.  By the Stanton-White
rim hook correspondence \cite{SW}, this latter number is equal to 
the number of pairs $(P,Q)$, where 
$P = (P_1, \dots, P_m)$ and $Q = (Q_1, \dots, Q_m)$
are $m$-tuples of semistandard tableaux with
$P_i$ having the same shape as $Q_i$ for all $i$ and such
that $P$ and $Q$ have contents
$\frac{m}{\ell} \mu^{1/k}$ and
$\frac{m}{k} \nu^{1/\ell}$.
By RSK insertion, this enumeration is again equal to the
number of sequences $(A_1, \dots, A_m)$
of $\frac{\ell(\mu) m}{ \ell} \times \frac{\ell(\nu) m}{k}$ 
$\mathbb{N}$-matrices with row vectors summing to
$\frac{m}{\ell} \mu^{1/k}$ and column vectors
summing to
$\frac{m}{k} \nu^{1/\ell}$.  An analysis of 
fundamental domains under the action of row and column
shifts shows that sequences of matrices as above are 
in bijection with $\ell(\mu) \times \ell(\nu)$ matrices $A$ with row
vector $\mu$ and column vector $\nu$ which are fixed under
$\ell(\mu)/k$-fold row rotation and $\ell(\nu)/\ell$-fold column rotation.
Up to sign, this proves Theorem 1.3.  

To make sure the sign in Theorem 1.3 is correct, 
we need to show that the expression
\begin{equation*}
\epsilon_n(\zeta,\zeta') \sum_{\lambda \vdash n} 
K_{\lambda,\mu}(\zeta) K_{\lambda,\nu}(\zeta')
\end{equation*}
is nonnegative.
By
Theorem 2.2 we need
only check that 
\begin{equation*}
\epsilon_n(\zeta,\zeta') = 
(-1)^{((k-1)\frac{n}{k} + (\ell-1)\frac{n}{\ell})}.
\end{equation*}
This is a routine exercise.
 
\end{proof}

In order to prove Theorem 1.4 we will need a 
pair of commutativity results
regarding the raising and lowering operators and the involution 
$\omega$.

\begin{lem}
1.  If $k$ is odd, we have that $\omega \phi_k = \phi_k \omega$ and
$\omega \psi^k = \psi^k \omega$. \\
2.  If $\ell > k$, we have the relation $\phi_{2^{\ell}} \omega \psi^{2^{k}} =
\phi_{2^{k}} \omega \psi^{2^{k}} \phi_{2^{\ell - k}}$. \\
3.  For any $\ell > 0$ and any composition $\mu$ such that $2^{\ell} | |\mu|$, we have that 
$\phi_{2^{\ell}} \omega (h_{\mu}) = (-1)^{\frac{|\mu|}{2^{\ell}}} \omega \phi_{2^{\ell}} (h_{\mu})$.
\end{lem}

\begin{proof}
The operator relations 1 and 2 can both be checked on the power sum functions $\{ p_n \}$ using the 
identity $\omega(p_n) = (-1)^{n-1} p_n$ together with the fact that $\omega$, the raising operators, and the lowering operators are all ring maps.

For 3, we again appeal to Formula 13 of \cite{LLTRibbon} to get that the evaluation
$\phi_a(e_N)$ of the lowering operator $\phi_a$ on the elementary symmetric function
$e_N$ is equal to $(-1)^{\frac{N}{a}(a-1)} e_{\frac{N}{a}}$ if $a | N$ and 0 otherwise for
any $a, N \geq 0$.  Here we have used that the $a$-core of the partition $(1^N)$ is empty if and only if $a | N$, in which case the $a$-quotient of $(1^N)$ is $((1^{\frac{N}{a}}), \emptyset, \dots, \emptyset)$ and the $a$-sign of $(1^N)$ is $(-1)^{\frac{N}{a}}$.  Using this evaluation, the desired identity can be proven using the fact that $\phi_{2^{\ell}}$ and $\omega$ are ring maps.
\end{proof}

\begin{proof} (of Theorem 1.4)
Without loss of generality, we may again assume that $\mu$ and
$\nu$ are strict.
Fix divisors $k | \frac{\ell(\mu)}{a}$ and $\ell | \frac{\ell(\nu)}{b}$, 
where each part of $\mu$ has multiplicity divisible by $k$
and each part of $\nu$ has multiplicity divisible by $\ell$.  
Let $\zeta$ and $\zeta'$ be roots of unity of multiplicative orders
$k$ and $\ell$.
Recalling that $\omega(s_{\lambda}) = s_{\lambda'}$,
up to sign we were interested in expressions like 
\begin{equation*}
\sum_{\lambda \vdash n} 
K_{\lambda',\mu}(\zeta) K_{\lambda,\nu}(\zeta') = 
\langle \omega (Q_{\mu}(x ; \zeta)), Q_{\nu} (x ; \zeta') \rangle.
\end{equation*}
Applying Theorem 2.2 we see that, up to the sign $(-1)^{\frac{n}{k}(k-1) + \frac{n}{\ell}(\ell-1)}$, the above expression is equal to
\begin{equation*}
\langle \omega \psi^k (h_{\mu^{1/k}}),
\psi^{\ell} (h_{\nu^{1/\ell}}) \rangle.
\end{equation*}
Let $m$ be the greatest common divisor of $k$ and $\ell$.  We consider several cases depending on the parities of $k$ and $\ell$.

If $k$ and $\ell$ are both odd, we can use Part 1 of Lemma 2.3 together with Lemma 2.1 to derive the identity 
\begin{equation*}
\langle \omega \psi^k (h_{\mu^{1/k}}),
\psi^{\ell} (h_{\nu^{1/\ell}}) \rangle = \langle \omega \psi^m (h_{\frac{m}{\ell} \mu^{1/k}}),
\psi^m (h_{\frac{m}{k} \nu^{1/\ell}}) \rangle.
\end{equation*}

If at least one of $k$ and $\ell$ are even, since $\omega$ is involutive and an isometry with respect to the Hall inner product, we can assume that $\frac{k}{m}$ is odd.  If both $k$ and $\ell$ are even, we can use Parts 1 and 2 of Lemma 2.3 together with Lemma 2.1 to again show that
\begin{equation*}
\langle \omega \psi^k (h_{\mu^{1/k}}),
\psi^{\ell} (h_{\nu^{1/\ell}}) \rangle = \langle \omega \psi^m (h_{\frac{m}{\ell} \mu^{1/k}}),
\psi^m (h_{\frac{m}{k} \nu^{1/\ell}}) \rangle.
\end{equation*}

However, if $k$ is odd and $\ell$ is even, assuming as before the $\frac{k}{m}$ is odd, we use Parts 1 and 3 of Lemma 2.3 together with Lemma 2.1 to show that
\begin{equation*}
\langle \omega \psi^k (h_{\mu^{1/k}}),
\psi^{\ell} (h_{\nu^{1/\ell}}) \rangle = 
(-1)^{\frac{n}{\ell}}
\langle \omega \psi^m (h_{\frac{m}{\ell} \mu^{1/k}}),
\psi^m (h_{\frac{m}{k} \nu^{1/\ell}}) \rangle.
\end{equation*}

Regardless of the parities of $k$ and $\ell$,
consider the Hall inner product
\begin{equation*}
 \langle \omega \psi^m (h_{\frac{m}{\ell} \mu^{1/k}}),
\psi^m (h_{\frac{m}{k} \nu^{1/\ell}}) \rangle.
\end{equation*}
Formula 17 of \cite{LLTRibbon} again allows us to perform the raising operator
evaluations
\begin{equation*}
\psi^m (h_{\alpha} ) = \sum_T \epsilon_m(T) s_{sh(T)},
\end{equation*}
and we have that $\omega (s_{\lambda}) = s_{\lambda'}$ for any partition $\lambda$.  
In addition, given any partition $\lambda$ with empty $m$-core, we have that the 
$m$-signs of $\lambda$ and $\lambda'$ are related by 
\begin{equation*}
\epsilon_m(\lambda) = (-1)^{(m-1)\frac{|\lambda|}{m}} \epsilon_m(\lambda').
\end{equation*}
Therefore, the Hall inner product of interest 
is equal to $(-1)^{(m-1)\frac{mn}{kl}}$ times the number of 
pairs $(P, Q)$ of $m$-tuples $P = (P_1, \dots, P_m)$ and $Q = (Q_1, \dots, Q_m)$ of 
SSYT such that $P$ has content $\frac{m}{\ell}\mu^{1/k}$,
$Q$ has content $\frac{m}{k}\nu^{1/{\ell}}$, and the shape of $P_i$ is
the \emph{conjugate} of the shape of $Q_i$ for all $i$.  By the dual RSK 
algorithm, this is the number of $m$-tuples $(A_1, \dots, A_m)$ of 
$\frac{\ell(\mu)m}{\ell} \times \frac{\ell(\nu)m}{k}$ $0,1$-matrices with row vectors summing to $\frac{m}{\ell}\mu^{1/k}$ and column vectors summing to
 $\frac{m}{k}\nu^{1/{\ell}}$.  Again, an elementary analysis of  fundamental domains implies that such $m$-tuples are in bijective correspondence with the fixed point set of interest.

To check that the sign in Theorem 1.4 is correct, we check that the expression
\begin{equation*}
\delta_n(\zeta,\zeta')
\sum_{\lambda \vdash n} 
K_{\lambda',\mu}(\zeta) K_{\lambda,\nu}(\zeta')
\end{equation*}
is nonnegative.  This is a routine case by case check depending on the parities of the numbers
$k, \ell, \frac{n}{k},$ and $\frac{n}{\ell}$.
\end{proof}

\section{Proofs of Theorems 1.3 and 1.4 using Representation Theory}

In this section we use results about the graded characters of DeConcini-Procesi modules \cite{DP} to sketch a representation theoretic proof of Theorems 1.3 and 1.4
up to modulus.  The author is grateful to Victor Reiner for outlining this argument.  

Given any integer $n > 0$, let $X_n$ denote the variety of complete flags
$0 = V_0 \subset V_1 \subset \dots \subset V_n = \mathbb{C}^n$ in $\mathbb{C}^n$ with $\dim V_i = i$.  For any composition $\mu \models n$, let 
$u \in GL_n(\mathbb{C})$ be an
$n \times n$ unipotent complex matrix with Jordan block decomposition given
by $\mu$.  The subset $X_\mu \subseteq X_n$ of flags stabilized by the action of $u$ is a subvariety of $X_n$ and Springer \cite{Springer} showed that 
the cohomology ring $H^{*}(X_{\mu})$  
carries a natural graded representation of $S_n$.
It turns out that 
$H^i(X_{\mu}) = 0$ for odd $i$, so one defines a graded $S_n$-module 
$R_{\mu} := \bigoplus_{d \geq 0} R_{\mu}^d$, with $R_{\mu}^d := H^{2d}(X_{\mu})$.  The graded character $\chr _q R_{\mu}$ is the symmetric function
$\chr _q R_{\mu} = \sum_{d \geq 0} q^d \chr R_{\mu}^d$, where $\chr R_{\mu}^d$ is the Frobenius character of the $S_n$-module $R_{\mu}^d$.  

Define the \emph{modified Kostka-Foulkes polynomial} $\widetilde{K}_{\lambda,\mu}(q) \in 
\mathbb{N}[q]$ to be the generating function for the cocharge statistic on SSYT of shape
$\lambda$ and content $\mu$:
\begin{equation*}
\widetilde{K}_{\lambda,\mu}(q) := \sum_{T} q^{cocharge(T)}.
\end{equation*}
For $\mu$ a partition,
the modified Kostka-Foulkes polynomials are related to the ordinary Kostka-Foulkes polynomials
by $K_{\lambda,\mu}(q) = q^{n(\mu)} \widetilde{K}_{\lambda,\mu}(\frac{1}{q})$, where
$n(\mu) = \sum_i (i-1)\mu_i$.  
The \emph{modified Hall-Littlewood polynomials} $\widetilde{Q}_{\mu}(x;q)$ 
for $\mu \models n$ a composition
are given by
\begin{equation*}
\widetilde{Q}_{\mu}(x;q) := \sum_{\lambda \vdash n} \widetilde{K}_{\lambda,\mu}(q) s_{\lambda}(x).
\end{equation*}
Garsia and Procesi \cite{GarP} proved 
that the graded character of the module $R_{\mu}$ is equal to the 
modified Hall-Littlewood polynomial: 
$\chr_q R_{\mu} = \widetilde{Q}_{\mu}(x;q)$.
For any number $\ell > 0$, we can regard $R_{\mu}$ as a graded 
$S_n \times \mathbb{Z}_{\ell}$-module by letting the cyclic group $\mathbb{Z}_{\ell}$ act on
the graded component $R_{\mu}^d$ by scaling by a factor of $e^{\frac{2 \pi i d}{\ell}}$.

Suppose now that the composition $\mu \models n$ has cyclic symmetry of order $a | \ell(\mu)$.
Let $Y_{\mu}$ be the set of all words $(w_1, \dots, w_n)$ of length $n$ and content $\mu$.  Then $Y_{\mu}$ is naturally a $S_n \times \mathbb{Z}_{\ell(\mu)/a}$-set, where the symmetric group $S_n$ acts on the indices and the cyclic group $\mathbb{Z}_{\ell(\mu)/a}$ acts on the letter values, sending 
$i$ to $i+a$ mod $\ell(\mu)$.  
The vector space $\mathbb{C}[Y_{\mu}]$ is therefore a module over 
$S_n \times \mathbb{Z}_{\ell(\mu)/a}$ by linear extension.
The following module isomorphism is a remarkable result of Morita and Nakajima.

\begin{thm} \cite[Theorem 13]{MNSym}
Let $\mu \models n$ by a composition with cyclic symmetry $a | \ell(\mu)$.
We have an isomorphism of $S_n \times \mathbb{Z}_{\ell(\mu)/a}$-modules
\begin{equation*}
R_{\mu} \cong \mathbb{C}[Y_{\mu}].
\end{equation*}
\end{thm}    

Morita and Nakajima proved this result by comparing the characters of the 
modules in question.  Morita  \cite[Theorem 4]{Morita} 
gave another character theoretic proof using the plethystic operators $\psi^k$ and $\phi_k$ in Section 3 of this paper.  Shoji \cite{Shoji} proved a generalization of this result to other types in which one replaces the variety $X_{\mu}$ with the variety of Borel subgroups containing a unipotent element $u$ of a simple algebraic group $G$ over $\mathbb{C}$.  

Now suppose that we are given two compositions $\mu, \nu \models n$.  
Elements of the product $Y_{\mu} \times Y_{\nu}$ can be thought of as 
$2 \times n$ matrices
$\begin{pmatrix}
a_{11} & a_{12} & \dots & a_{1n} \\
a_{21} & a_{22} & \dots & a_{2n}
\end{pmatrix}$ of letters such that the content of the word
$a_{11}a_{12} \dots a_{1n}$ is equal to $\mu$ and the content of the word
$a_{21}a_{22} \dots a_{2n}$ is equal to $\nu$.  The product $S_n \times S_n$
 of symmetric groups acts on these matrices by independent permutation of the
 indices in the top and bottom rows.  If in addition the compositions $\mu$
and $\nu$ have cyclic symmetries of orders $a | \ell(\mu)$ and $b | \ell(\nu)$,
then the product set $Y_{\mu} \times Y_{\nu}$ carries an action of 
$S_n \times \mathbb{Z}_{\ell(\mu)/a} \times S_n \times \mathbb{Z}_{\ell(\nu)/b}$,
where the cyclic groups act by modular addition on the letter values.  

As a direct consequence of Theorem 3.1 we have that
\begin{equation*}
R_{\mu} \otimes_{\mathbb{C}} R_{\nu} \cong \mathbb{C}[Y_{\mu} \times Y_{\nu}]
\end{equation*} 
as modules over the group $S_n \times \mathbb{Z}_{\ell(\mu)/a} \times S_n \times \mathbb{Z}_{\ell(\nu)/b}$, where the module on the left hand side is
bigraded.
Considering the diagonal embedding $S_n \hookrightarrow S_n \times S_n$ given by $w \mapsto (w,w)$, restricting the above isomorphism yields an isomorphism 
\begin{equation*}
R_{\mu} \otimes_{\mathbb{C}} R_{\nu} \cong \mathbb{C}[Y_{\mu} \times Y_{\nu}]
\end{equation*}
of $S_n \times \mathbb{Z}_{\ell(\mu)/a} \times \mathbb{Z}_{\ell(\nu)/b}$-modules.  
Viewing elements of $Y_{\mu} \times Y_{\nu}$ as $2 \times n$ matrices, the action of $S_n$ on the right hand side is induced by its natural action on
matrix columns.  
Finally, if 
$\epsilon$ is any irreducible character of $S_n$, we may restrict the above
isomorphism to its $\epsilon$-isotypic component to get an isomorphism 
\begin{equation*}
[R_{\mu} \otimes_{\mathbb{C}} R_{\nu}]^{\epsilon} \cong \mathbb{C}[Y_{\mu} \times Y_{\nu}]^{\epsilon}
\end{equation*}
of modules over 
$\mathbb{Z}_{\ell(\mu)/a} \times \mathbb{Z}_{\ell(\nu)/b}$, where the 
exponential notation denotes taking isotypic components and 
the left hand side is bigraded with the cyclic groups acting by scaling 
by a root of unity in each grade.  At least
up to modulus, Theorems 1.3 and 1.4 can be deduced
from specializing $\epsilon$ to the trivial and sign characters of $S_n$, respectively.

Suppose first that $\epsilon = \triv$ is the trivial character of $S_n$.  Then the isotypic component $\mathbb{C}[Y_{\mu} \times Y_{\nu}]^{\triv} =
\mathbb{C}[Y_{\mu} \times Y_{\nu}]^{S_n}$ has a natural basis given by sums over orbits of the action of $S_n$ on $Y_{\mu} \times Y_{\nu}$.  Each of these orbits has a unique representative of the form
$\begin{pmatrix}
a_{11} & a_{12} & \dots & a_{1n} \\
a_{21} & a_{22} & \dots & a_{2n}
\end{pmatrix}$, where the biletters 
$\begin{pmatrix}
a_{1i} \\
a_{2i} \end{pmatrix}$ are in lexicographical order.  Such orbit representatives are in
natural bijection with $\mathbb{N}$-matrices with row content $\mu$ and 
column content $\nu$.  It is easy to see that the action of the cyclic group
product $\mathbb{Z}_{\ell(\mu)/a} \times \mathbb{Z}_{\ell(\nu)/b}$ is given
by $a$-fold row and $b$-fold column rotation.  Therefore, the number of
fixed points of a group element $g \in \mathbb{Z}_{\ell(\mu)/a} \times \mathbb{Z}_{\ell(\nu)/b}$ in the action of Theorem 1.3 is equal to the trace of $g$ on
$\mathbb{C}[Y_{\mu} \times Y_{\nu}]^{\triv}$ and therefore is also equal
to the trace of $g$ on $[R_{\mu} \otimes_{\mathbb{C}} R_{\nu}]^{\triv}$. 
This latter trace can be identified with a polynomial evaluation at roots
of unity by considering the bigraded Hilbert series of 
$[R_{\mu} \otimes_{\mathbb{C}} R_{\nu}]^{\triv}$, proving Theorem 1.3 up to modulus.

To prove Theorem 1.4, we instead focus on the \emph{sign} character 
$\epsilon = \sign$ of the symmetric group $S_n$.  The isotypic component
$\mathbb{C}[Y_{\mu} \times Y_{\nu}]^{\sign}$ has as basis the set of 
$S_n$-antisymmetrized sums over the element of the set 
$Y_{\mu} \times Y_{\nu}$.  Representing elements of $Y_{\mu} \times Y_{\nu}$ as $2 \times n$ matrices, since antisymmetrization kills any matrix with repeated biletters, these basis elements are in natural bijection with
$0,1-$matrices of row content $\mu$ and column content $\nu$.  The cyclic group product $\mathbb{Z}_{\ell(\mu)/a} \times
\mathbb{Z}_{\ell(\nu)/b}$ acts on this basis by $a$-fold row and $b$-fold
column rotation, up to a plus or minus sign which arises from antisymmetrization
and sorting biletters into lexicographical order.  It is fairly easy to see
that up to sign the number of fixed points of a group element
$g \in \mathbb{Z}_{\ell(\mu)/a} \times
\mathbb{Z}_{\ell(\nu)/b}$ is the absolute value of the trace of $g$ on 
$\mathbb{C}[Y_{\mu} \times Y_{\nu}]^{\sign}$.  This latter number is also
the absolute value of the trace of $g$ on 
$[R_{\mu} \otimes_{\mathbb{C}} R_{\nu}]^{\sign}$.  This trace can be
identified with a polynomial evaluation by considering bigraded Hilbert 
series as in the case of the trivial isotypic component.  Up to modulus, 
this verifies Theorem 1.4.

\section{Acknowledgements}
The author is grateful to Victor Reiner, Dennis Stanton, and Dennis White for helpful conversations.

\bibliography{../bib/my}

\end{document}